\documentclass{article}
\usepackage{indentfirst}
\usepackage{amsthm}
\usepackage{amsfonts}
\usepackage{amssymb}
\usepackage{overpic}

\newcommand{\COLORON}{1}
\newcommand{\NOTESON}{0}
\newcommand{\Debug}{0}

\newcommand{\aut}{\mathrm{Aut}}
\newcommand{\CQ}{\ensuremath{\cc_\Q}}
\newcommand{\ILC}{invariant under local complementation}

\newcommand{\dsg}{$d$-sphere graph}

\usepackage[usenames]{color} 
\usepackage{amsthm,amssymb,amsmath,bbm,enumerate,graphicx,epsf,stmaryrd,accents}
\usepackage[bookmarks, colorlinks=false, breaklinks=true]{hyperref} %REMOVE FOR ARXIV 

\usepackage{authblk}

\newcommand{\comment}[1]{}
\newcommand{\COMMENT}[1]{}

\definecolor{darkgray}{rgb}{0.3,0.3,0.3}
\newcommand{\defi}[1]{{\color{darkgray}\emph{#1}}}

\newcommand{\acknowledgement}{\section*{Acknowledgement}}

\comment{
	\begin{lemma}\label{}	
\end{lemma}
\begin{proof}

\end{proof}

\begin{theorem}\label{}
\end{theorem} 
\begin{proof} 	

\end{proof}

}

\newtheorem{proposition}{Proposition}[section]

\newtheorem{theorem}[proposition]{Theorem}
\newtheorem{corollary}[proposition]{Corollary}

\newtheorem{lemma}[proposition]{Lemma}
\newtheorem{observation}[proposition]{Observation}
\newtheorem{conjecture}{{Conjecture}}[section]

\newtheorem{problem}[conjecture]{{Problem}}

\newtheorem{examp}[proposition]{Example}%[section]

\newtheorem{remark}[proposition]{Remark}

\newcommand{\FIG}{0}

\ifnum \NOTESON = 1 \newcommand{\note}[1]{ 

\hspace*{-30pt}
	{\color{blue}  NOTE: \color{Turquoise}{\small  \tt \begin{minipage}[c]{1.1\textwidth}  #1 \end{minipage} \ignorespacesafterend }} 
	
	}
\else \newcommand{\note}[1]{} \fi

\newcommand{\afsubm}[1]{ \ifnum \Debug = 1 {\mymargin{#1}}
\fi} %For notes on after-submission changes

\ifnum \Debug = 1 
\else  \fi

\ifnum \FIG = 1 \newcommand{\fig}[1]{Figure ``{#1}''}
\else \newcommand{\fig}[1]{Figure~\ref{#1}} \fi

\ifnum \FIG = 1 
\else  \fi

\ifnum \Debug = 1 \usepackage[notref,notcite]{showkeys}
\fi

\ifnum \COLORON = 0 \renewcommand{\color}[1]{}
\fi

\newcommand{\N}{\ensuremath{\mathbb N}}
\newcommand{\R}{\ensuremath{\mathbb R}}

\newcommand{\Q}{\ensuremath{\mathbb Q}}
\newcommand{\BS}{\ensuremath{\mathbb S}}

\newcommand{\cc}{\ensuremath{\mathcal C}}

\newcommand{\ch}{\ensuremath{\mathcal H}}
\newcommand{\ci}{\ensuremath{\mathcal I}}

\newcommand{\ck}{\ensuremath{\mathcal K}}

\newcommand{\sm}{\backslash}

\makeatletter
\DeclareRobustCommand{\cev}[1]{%
  \mathpalette\do@cev{#1}%
}
\newcommand{\do@cev}[2]{%
  \fix@cev{#1}{+}%
  \reflectbox{$\m@th#1\vec{\reflectbox{$\fix@cev{#1}{-}\m@th#1#2\fix@cev{#1}{+}$}}$}%
  \fix@cev{#1}{-}%
}
\newcommand{\fix@cev}[2]{%
  \ifx#1\displaystyle
    \mkern#23mu
  \else
    \ifx#1\textstyle
      \mkern#23mu
    \else
      \ifx#1\scriptstyle
        \mkern#22mu
      \else
        \mkern#22mu
      \fi
    \fi
  \fi
}

\makeatother
\newcommand{\seq}[1]{\ensuremath{(#1_n)_{n\in\N}}} 

\newcommand{\g}{\ensuremath{G\ }}
\newcommand{\G}{\ensuremath{G}}

\newcommand{\Cg}{Cayley graph}

\newcommand{\Lr}[1]{Lemma~\ref{#1}}

\newcommand{\Tr}[1]{Theorem~\ref{#1}}

\newcommand{\Sr}[1]{Section~\ref{#1}}

\newcommand{\Prr}[1]{Pro\-position~\ref{#1}}

\newcommand{\Cr}[1]{Corollary~\ref{#1}}

\newcommand{\Or}[1]{Observation~\ref{#1}}

\newcommand{\Rr}[1]{Remark~\ref{#1}}

\renewcommand{\iff}{if and only if}
\newcommand{\fe}{for every}

\newcommand{\st}{such that}

\newcommand{\obda}{without loss of generality}

\newcommand{\wrt}{with respect to}

\newcommand{\labtequ}[2]{%\labtequc{#1}{#2}}
 \begin{equation} \label{#1} 	\begin{minipage}[c]{0.9\textwidth}  #2 \end{minipage} \ignorespacesafterend \end{equation} }

\newcommand{\mymargin}[1]{% <- dieses % verhindert ein ungewolltes Leerzeichen
 \ifnum \Debug = 1
  \marginpar{%
    \begin{minipage}{\marginparwidth}\small%
      \begin{flushleft}%
        {\color{blue}#1}%
      \end{flushleft}%
   \end{minipage}%
  }%
 \fi
}%

\newcommand{\extras}[1]{% <- dieses % verhindert ein ungewolltes Leerzeichen
 \ifnum \Debug = 1
\section{Extras} #1
 \fi
}%

\newcommand{\mySection}[2]{}

\newcommand{\DB}{\cite{diestelBook25}}
\begin{document}
	\title{Circle graphs and the automorphism group of the circle}
	
\author{Agelos Georgakopoulos\thanks{Supported by EPSRC grants EP/V048821/1 and EP/V009044/1.}}
\affil{  {Mathematics Institute}\\
 {University of Warwick}\\
  {CV4 7AL, UK}}
%\author{}

\date{\today}

\maketitle

\begin{abstract}
We prove that $\aut(\BS^1)$ coincides with the automorphism group of the \emph{circle graph} \cc, i.e.\ the intersection graph of the family of chords of $\BS^1$. 

We prove that the countable subgraph of \cc\ induced by the rational chords is a strongly universal element of the family of circle graphs, and that it is invariant under local complementation. The only other known connected graphs that have the latter property are $K_2$ and the Rado graph.
\end{abstract}

local complementation, Ivanov's metaconjecture, sphere graph. \\


\section{Introduction}

It is well-known that for every group $\Gamma$ there is a graph \g \st\ $\Gamma$  coincides with the automorphism group $\aut(G)$ of \G\ \cite{dGrGro,SabGra}. The construction of such a graph \g is rather ad-hoc in this generality: it boils down to decorating the edges of a \Cg\ of $\Gamma$ in a way that eliminates unwanted symmetries. But it is sometimes desirable to find such a \g with additional properties that will help us draw conclusions about $\Gamma$. A prime example is a well-known theorem of Ivanov \cite{Ivanov}, stating that the mapping class group $Mod(S)$ of each topological surface $S$ of finite type and genus at least 2 coincides with  $\aut(C(S))$, where $C(S)$ stands for the so-called {curve graph} on $S$. This theorem has found many applications, establishing $C(S)$ as a central tool in the study of  $Mod(S)$; see \cite{DiKoGoMod} and references therein. Ivanov stated the influential \defi{metaconjecture} that one can replace $C(S)$ by any rich enough graph naturally associated to $S$ in the above theorem \cite{IvaFif}. 

Our first main result is similar in spirit to Ivanov's metaconjecture, and provides a natural graph $\cc$ such that $\aut(\cc)$ coincides with $\aut(\BS^1)$, i.e. the group of homeomorphisms of the unit circle. The group $\aut(\BS^1)$ is on its own a rich and classical topic, a melting pot of algebra, analysis, geometry \& dynamics. Several surveys and monographs have been devoted to it, and we refer to them for further study \cite{Beklaryan,GhyGro,Navas}. Several other important (families of) groups are subgroups of $\aut(\BS^1)$, e.g.\ $PSL(2, \R)$ and Thompson's group.

\medskip
The aforementioned graph $\cc$ is what we will call \defi{the circle graph}: its vertices are the chords of $\BS^1$, %(equivalently, the unordered pairs $a,b\in \BS^1, a\neq b$), 
and two chords form an edge of \cc\ whenever they intersect (either at a point of $\BS^1$, or one inside it).

Note that each homeomorphism $g$ of $\BS^1$ canonically induces an automorphism $\pi(g)$ of \cc. Our first theorem provides a converse that makes it similar to Ivanov's aforementioned theorem: 

\begin{theorem} \label{thm aut}
The map $\pi$ is an isomorphism from $\aut(\BS^1)$ onto $\aut(\cc)$.
\end{theorem}

For the rest of this introduction we will focus on the \defi{rational circle graph \CQ}, i.e.\ the (countable) subgraph of  \cc\ induced by the chords with both end-points being rational. The proof of \Tr{thm aut} adapts to show that $\aut(\CQ)$ coincides with the subgroup of $\aut(\BS^1)$ that fixes \Q\ (\Cr{cor cq}).

But \cc\ and \CQ\ are interesting from a purely graph theoretic perspective as well. To begin with, we remark that they are vertex-transitive by construction. However, we will observe that they are not edge-transitive (\Rr{rem et}).  

Their finite induced subgraphs, which are called \defi{circle graphs}, are studied for a variety of reasons. They are similar to interval graphs as well as circle-arc graphs, which are classical topics. (In fact I expect many of the constructions and results of this note to extend to other such graph classes.) Charactering such classes in terms of their forbidden induced subgraphs is a mainstream question, which remains open for circle graphs; see \cite{DuGrSaStr} for a survey. Nevertheless, Bouchet \cite{BouCir} obtained a characterization in terms of just three forbidden \defi{vertex minors}. We say that a graph $H$ is a vertex minor of a graph \G, if $H$ can be obtained from an induced subgraph $G'$ of \g by a sequence of \defi{local complementation}, where the local complementation \defi{$G'_v$} of $G'$ at a vertex $v$ is obtained by swapping edges and non-edges in the neighbourhood $N(v)$. We will show that Bouchet's result extends to the infinite case (\Or{obs}).

There has been renewed interest in circle graphs recently, due to a conjectural structure theorem of  Geelen \cite{McCarty}  for vertex minors analogous to that of Robertson \& Seymour \cite{GM17} for minors. The circle graphs thereby play an important role analogous to that of planar graphs in the theory of Robertson \& Seymour. Our next result is

\begin{theorem} \label{thm locom}
\CQ\ is invariant under local complementation. 
\end{theorem}

This is a rather rare property of graphs, and the only other connected graphs I know that enjoy it are the Rado graph (\Or{Rado}) and $K_2$.

\medskip
Recall that a graph $U$ is \defi{strongly universal} for a graph class $\mathcal{C}\ni U$, if every $G\in \mathcal{C}$ is an induced subgraph of $U$. Understanding which classes of graphs admit strongly universal elements is a classical topic in infinite graph theory, see \cite{KomPacUni} for a survey. The best known such construction is the \defi{Rado graph}, a strongly universal element of the class of all countable graphs. But there are natural classes of graphs that fail to have a universal element even under much weaker containment notions, e.g.\ the rayless graphs \cite{universal}. Our next main result is %that \CQ\ is a strongly universal element of the class of countable circle graphs. In fact we will prove something stronger: 

\begin{theorem} \label{thm univ}
\CQ\ is a strongly universal element for the class of countable circle graphs. % $\cc'$. Moreover, $\cc'$ can be chosen to be vertex transitive, and with no incident chords, and (therefore) invariant under local-complementation. 
\end{theorem}

\begin{problem}
Is there a non-empty, connected, countably infinite (vertex-transitive) graph which is invariant under local complementation and is not isomorphic to the Rado graph or $\cc_\Q$? Are there infinitely many such graphs? 
\end{problem}

In particular, is $\cc_\Q$ the unique, up to isomorphism, non-empty, connected, countable circle graph which is invariant under local complementation? It could also be interesting to look for finite (circle) graphs that are invariant under local complementation. I do not know of any connected example with more than one edge.

\medskip
We prove \Tr{thm aut} in \Sr{sec aut}. The  aforementioned facts about  \CQ\ are proved in \Sr{sec CQ}. In \Sr{sec rado} we prove that the  Rado graph is invariant under local complementation. In \Sr{sec prob} we consider a higher-dimensional analogue of circle graphs and conclude with related open problems.

\mymargin{ Notes: 
$\cc$ has diameter 2, but its induced subgraphs may have interesting geometry. James Davies said: ``any class of graphs forbidding a finite graph $H$ as a vertex-minor is $f(H)$-quasi-isometric to a graph of bounded tree width. In fact circle graphs are quasi-isometric to cacti. Complements of circle graphs are  quasi-isometric to graphs of bounded tree-width.'' }

%\begin{proposition} \label{VT}
%\cc\ is vertex-transitive... (but not edge-transitive). 
%\end{proposition}

%invariant under local-complementation? 

\section{Preliminaries} \label{prel}
%... $\BS^d$, north pole... 

\subsection{Orders and the rationals} \label{sec ord}
%We define $\BS^1$  for the topological space  ...$\Q^1$...
The $d$-sphere $\BS^d$ is the set $\{x\in \R^d \mid ||x,(0,\ldots,0)|| = 1\}$, where $||\cdot||$ stands for euclidean distance. A \defi{chord} of $\BS^1$ is a line segment joining two distinct points of $\BS^1$.

We let \defi{$\Q^1$} denote the `rational' points of $\BS^1$. The most convenient way to make this precise is by identifying $\BS^1$ with the space obtained from $[0,1]$ by identifying $0$ with $1$.

The ordering of the set \Q\ of rationals has the property that 
\labtequ{betw}{\fe\ $a,b\in \Q$ there is 
$c\in \Q$ with $a<c<b$.}
In fact this property determines the order type of $(\Q,<)$: any countable order $(Q,<')$ satisfying \eqref{betw} is order-isomorphic to $(\Q,<)$. This fact easily extends to cyclic orderings by considering $\Q^1$. 

\begin{proposition} \label{nestedness}
Let $\chi: \BS^1 \to \BS^1$ be a bijection that preserves nestedness of pairs of distinct points. Then $\chi$ is a homeomorphism.
\end{proposition}
%%%%%%%%%
\begin{proof}
We just need to check that $h'$ is continuous, and for this it suffices to check that it preserves the cyclic ordering of $\BS^1$ up to inversion. This is an easy and well-known consequence of the fact that $h'$ preserves nestedness of pairs of distinct points of $\BS^1$ \cite[\S 3.4]{CoxProjective}. 
\end{proof}

\subsection{Graphs} \label{sec gra}
We follow the terminology of Diestel \DB. We use $V(G)$ to denote the set of vertices, and $E(G)$ the set of edges of a graph \G. For $S\subseteq V(G)$, the subgraph $G[S]$ of \g \defi{induced} by $S$ has vertex set $S$ and contains all edges of \g with both end-vertices in $S$.

Let $G$ be a graph, and $v\in V(G)$. The neighbourhood $N(v)=N_G(v)$ of $v$ is the set of vertices $u$ \st\ $uv\in E(G)$. We write \defi{$G_v$} for the graph obtained from $G$ by complementing the subgraph induced by $N(v)$, i.e.\ by deleting all edges with both end-vertices in $N(v)$, and introducing all edges within $N(v)$ that were missing from \G. We say that $G_v$ is obtained from \g by \defi{local complementation} at $v$.

The \defi{intersection graph} of a family of sets $\ch$ is the graph with vertex set \ch\ in which two vertices form an edge whenever their intersection is non-empty.

\section{Proof of \Tr{thm aut}}  \label{sec aut}

Given $x\in \BS^1$, notice that the set of vertices $\{xy \mid y\in \BS^1 - x\}$ of \cc\ induces a clique \defi{$K_x$}. We call $K_x$ a  \defi{boundary clique} of \cc. Obviously, $K_x\neq K_y$ if $x\neq y$.
\begin{lemma} \label{lem bc}
Every $h\in \aut(\cc)$ maps each boundary clique onto a boundary clique.
\end{lemma}
%%%%%%%%%%%%%%%%%
\begin{proof}
Call two chords $C,D$ of $\BS^1$ \defi{incident}, if $C\cap D$ is a point of $\BS^1$. We claim that 
\labtequ{pres inc}{$h$ preserves pairs of incident chords of $\BS^1$.} 
For suppose not, and assume \obda\ that $C,D$ are not incident but $h(C),h(D)$ are. Since incident chords intersect, we have $h(C)h(D)\in E(\cc)$, and therefore $CD \in E(\cc)$ since $h$ is an automorphism. Thus $C,D$ intersect (at an interior point of the unit disc $\mathbb{D}$, see \fig{figChords}). Pick four chords $F_1,F_2,F_3,F_4$, one in each of the four intervals of $\BS^1 \sm (C \cup D)$, and additional chords $W_{ij}$ intersecting $F_i$ and $F_j$ \fe\ $i,j \in [4]$. Note that $W_{ij}$ must intersect at least one of $C,D$. Since $h(C),h(D)$ are incident, $\BS^1 \sm (h(C) \cup h(D))$ consists of three intervals, and one of them must contain at least two of the $h(F_i), \in [4]$, say $h(F_i)$ and $h(F_j)$. But then $h(W_{ij})$ does not intersect any of $h(C),h(D)$, a contradiction since $W_{ij}$ intersects at least one of $C,D$, and $h$ is an isomorphism of \cc. This proves \eqref{pres inc}.

\begin{figure} 
\begin{center}
\begin{overpic}[width=.95\linewidth]{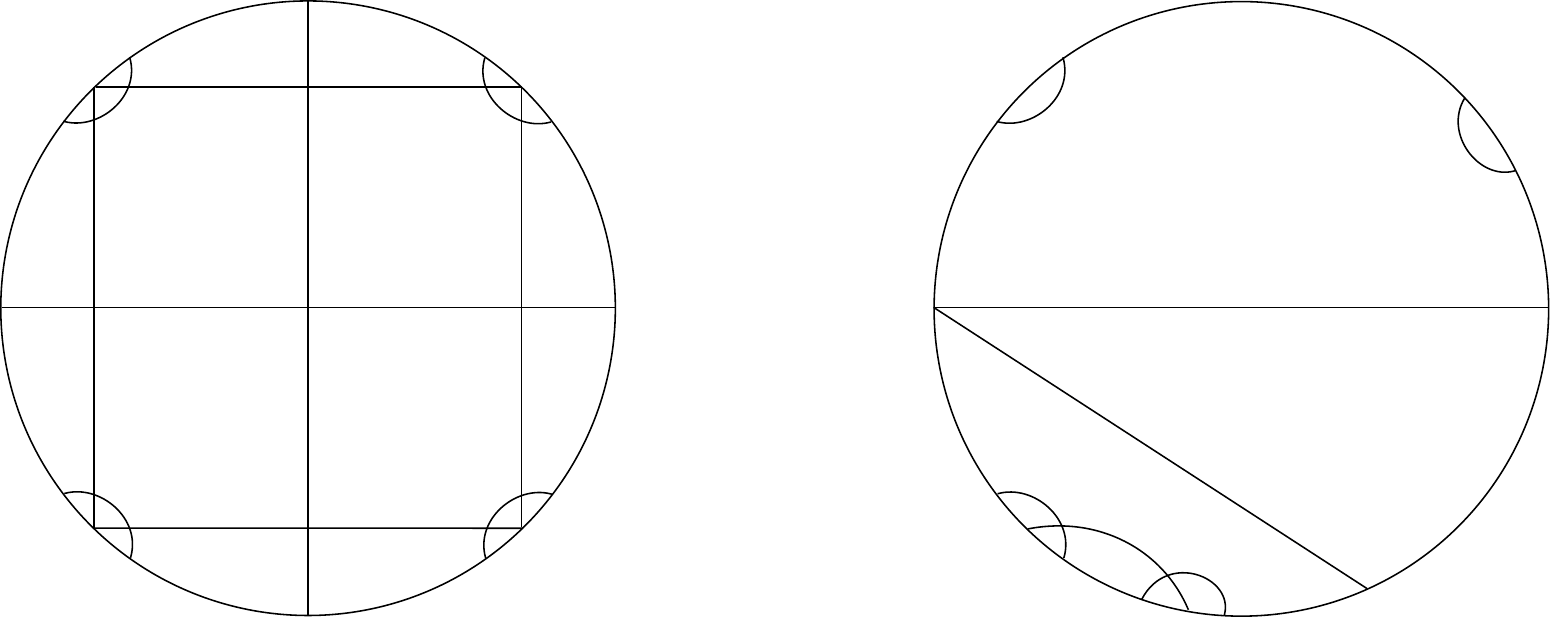} 
\put(21,25){$D$}
\put(13,21){$C$}
\put(72,21){$h(C)$}
\put(71,14){$h(D)$}
\put(2,34){$F_1$}
\put(2,3){$F_4$}
\put(35,34){$F_2$}
\put(33,3){$F_3$}
\put(12,35){$W_{12}$}
\put(57.5,5){$h(F_i)$}
\put(75,-3){$h(F_j)$}
\put(72,6){$W_{ij}$}
\put(48,21){$h$}
\put(46,18){$\longrightarrow$}

\end{overpic}
\end{center}
\caption{A contradiction to mapping non-incident chords $C,D$ to incident ones in the proof of \Lr{lem bc}.} \label{figChords}
\end{figure}
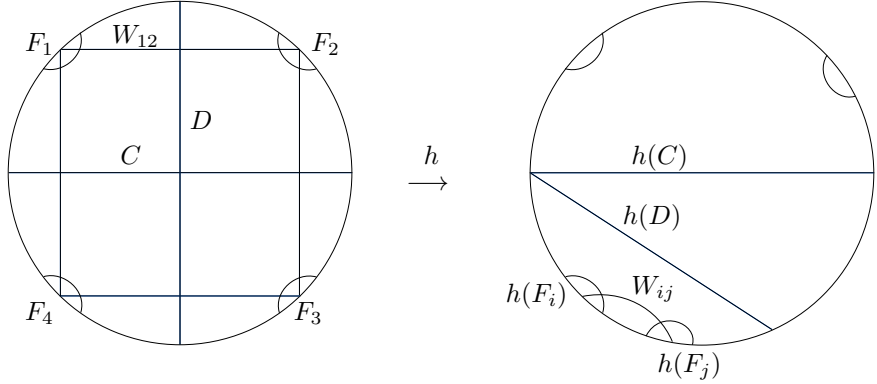

It follows  that $h$ preserves $n$-tuples of pairwise incident chords \fe\ $n\geq 3$, i.e.\ $n$-tuples of chords with a common endpoint on $\BS^1$. Since a boundary clique $K_x$ is such a tuple by definition, it follows that $h(K_x)$ is contained in $K_y$ for some $y\in \BS^1$. If $h(K_x)$ is a proper subset of $K_y$, then $h^{-1}(K_y)$ is not contained in $K_x$ but intersects it, which contradicts that $h^{-1}$ preserves  pairwise incident chords. Thus $h(K_x)=K_y$.
\end{proof}

\begin{remark} \label{rem et}
\eqref{pres inc} implies that \cc\ is not edge-transitive: no $h\in \aut(\cc)$ can map an edge of \cc\ comprising a pair of incident chords to an edge comprising a pair of intersecting non-incident chords.
\end{remark} 
%%%%%%%%%%%%%%%%%
\begin{proof}[Proof of \Tr{thm aut}]
It is straightforward to check that $\pi: \aut(\BS^1)\to \aut(\cc)$ is an injective homomorphism, and so it only remains to show that it is surjective. For this, given $h\in \aut(\cc)$, we will find $h'\in \aut(\BS^1)$ \st\ $\pi(h')= h$.

To define $h'$, we recall that, by \Lr{lem bc}, \fe\ $x\in \BS^1$, we have $h(K_x) = K_y$ for a unique $y\in \BS^1$. We define $h'$ by $x \mapsto y$. %the boundary 

We claim that $h': \BS^1 \to \BS^1$ is a bijection satisfying the assumption of  \Prr{nestedness}. % \eqref{pi h}. 

To see that $h'$ is injective, suppose $x\neq y\in \BS^1$, and $h'(x)=h'(z)=y$. Then $h(K_x) = h(K_z) = K_y$, which contradicts \Lr{lem bc} when applied to $h^{-1}$ and $y$ since $K_x\neq K_z$.

To see that $h'$ is surjective, pick $y\in \BS^1$, and note that $h^{-1}(K_y)$ coincides with some boundary clique $K_x$ using \Lr{lem bc} again, and so $h'(x)=y$.

Easily,  \fe\ $u,v,x,y\in \BS^1$, we have
\labtequ{pi h}{$h(uv)=xy$ \iff\ $h'(\{u,v\}) = \{x,y\}$}
because $K_u, K_v$ are the unique boundary cliques containing $uv$, and similarly for $ K_x, K_y, xy$. 
This immediately yields $\pi(h')= h$ as desired.

By \eqref{pi h}, $h'$ preserves nestedness of pairs of distinct points of $\BS^1$, and so it is a homeomorphism by \Prr{nestedness}.
\end{proof}

Let $\aut_\Q(\BS^1)$ denote the subgroup of $\aut(\BS^1)$ comprising those homeomorphisms that fix $\Q^1$. Our proof of \Tr{thm aut} adapts with minor changes to yield
\begin{corollary} \label{cor cq}
The restriction of $\pi$ to $\aut_\Q(\BS^1)$ is an isomorphism onto $\aut(\CQ)$.

\end{corollary}

\section{Properties of \CQ.}  \label{sec CQ}

In this section we prove the facts about \CQ\ mentioned in the introduction. We start with universality.

\begin{proof}[{Proof of \Tr{thm univ}} 
]
%First, we claim that $\cc_\Q$ (... the circle graph on \Q...) is a strongly universal countable circle graph. To prove this, l
Let \g be a countable circle graph, and let $v_0,v_1,\ldots$ be an enumeration of $V(G)$. We will realise \g as an induced subgraph of \CQ. 

For each $v_i$, let $p_i,q_i \in \BS^1$ be the two endpoints of the chord $v_i$. We can recursively map the $p_i,q_i$ to $\Q$ so that their cyclic-ordering is preserved. Indeed, having defined images $\pi(p_i),\pi(q_i)$ \fe\ $i<j$, we can use \eqref{betw} to define $\pi(p_j),\pi(q_j)$ so that the cyclic ordering of the $p_i,q_i$ coincided with that of $\pi(p_i),\pi(q_i)$. 

We can then map $V(G)$ to $V(\cc_\Q)$ by letting $\pi(v_i)$ be the chord spanned by $\pi(p_i),\pi(q_i)$. The subgraph of $\cc_\Q$ induced by $\pi(V(G))$ is then isomorphic to $G$, because $v_i,v_j$ intersect if and only if  $\pi(v_i),\pi(v_j)$ do by the construction of $\pi$. 
\end{proof}

Some of the ideas used in the proof of \Tr{thm univ} imply that being a circle graph is a finitary property in the following sense: 
\begin{observation} \label{obs}
Let \g be a countable graph every finite subgraph of which is a circle graph. Then \g is a circle graph.
\end{observation}
\begin{proof}[Proof (sketch)]
Let $v_0,v_1,\ldots$ be an enumeration of $V(G)$, and let $G_i:= G[\{v_0, \ldots, v_i\}]$ be the subgraph induced by the vertices up to $v_i$. Pick a realisation $R_i$ of each $G_i$ as a circle graph, and note that the ordering of the points of $\BS^1$ used by $R_i$ determines $G_i$. By an elementary compactness argument, there is a subsequence of \seq{R}\ along which these orderings converge, to a sub-order of $\Q^1$. This limit ordering defines a realisation of \g as a circle graph. We leave the details to the interested reader.
\end{proof}

In particular, Bouchet's characterization of the finite circle graphs in terms of forbidden vertex minors extends to the countably infinite case. 

%Recall that $\cc_\Q$ is vertex transitive \Prr{VT}. %Since $\cc_\Q$ has incident chords, 

\medskip
Next, we prove that \CQ\ is \ILC.

\begin{proof}[{Proof of \Tr{thm locom}}]
We construct an auxilliary circle graph $\cc'$ that has no incident chords and is isomorphic, as an abstract graph, to $\cc_\Q$. Thus it will suffice to prove that $\cc'$ is invariant under local complementation. 

For this, let $\ck \subset [0,1]$ be the complement of the standard Cantor set. Note that \ck\ is a countable set $\ci$ of disjoint intervals, which, when ordered by the natural ordering $<$ induced by that of $\R$, satisfy \eqref{betw}. Therefore, there is an order-preserving bijection $o: (\Q \cap (0,1)) \to \ci$. 
We define $\cc'$ from $\cc_\Q$ by `blowing-up' each $q\in \Q^1$ into $o(q) \cap \Q$, so that each chord in $V(\cc_\Q)$ incident with $q$ becomes incident with a distinct point in $o(q)$, and intersections of chords are preserved. To achieve this, we start by choosing, for each $q\in \Q^1$, an injection $o_q$ from $\Q^1 - q$  into the interval $o(q)$ \st\ $a<b$ \iff\ $o_q(a)< o_q(b)$ holds \fe\ $a,b\in \Q^1 - q$. For example, we could let $q'$ be the midpoint of the chord joining the two endpoints of $o_q$, and let $o_q(a)$ be the other endpoint of the chord of $\BS^1$ emanating from the midpoint of $o(a)$ and passing through $q'$ (\fig{figMidpoints}). %Easily, we can modify this construction in such a way that 
Given $p,q\in \BS^1$, let \defi{$c(p,q)$} denote the $p$--$q$~chord of $\BS^1$. Having chosen $o_q$ \fe\ $q\in \Q^1$, we then map each vertex $v=c(p,q)$ of $\cc_\Q$ to the chord $o(v):= c(o_q(p),o_p(q))$. Let $\cc'$ be the circle graph induced by  $o(V(\cc_\Q))$. It is straightforward to check that $o(v)$ intersects $o(v')$ \iff\ $v$ intersects $v'$. In other words, $\cc'$ is isomorphic to $\cc_\Q$. %; in particular, it is vertex-transitive. 
Moreover, no two vertices of $\cc'$ share an endpoint on $\BS^1$ by construction.

%\pdfFig{}{dfa}{}

\begin{figure} 
\begin{center}
\begin{overpic}[ width=.25\linewidth]{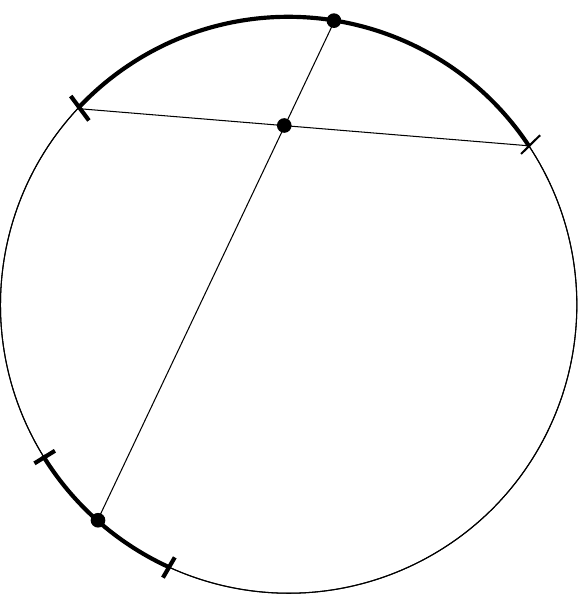} 
\put(23,104){$\leftarrow o(q) \rightarrow $}
\put(-3,2){$o(a)$}
\put(47,68){$q'$}
\put(53,86){$o_q(a)$}
\end{overpic}
\end{center}
\caption{Defining the  injection $o_q$ in the proof of \Tr{thm locom}.} \label{figMidpoints}
\end{figure}

It remains to check that $\cc'$ is invariant under local complementation. Pick a vertex $v=o(c(p,q))$ of $\cc'$, and let $H:= \cc'_v$.  Using a common trick, we can realise $H$ as a circle graph by flipping one of the two intervals $I$ of $\BS^1$ bounded by $o_q(p),o_p(q)$ (see e.g.\ \cite[Figure 1.3]{McCarty}). We can do so in such a way that $I\cap o(p)$ is mapped onto $I\cap o(q)$, and $\ck\cap I$ is mapped onto itself. But this realisation of $H$ coincides with $\cc'$: the two graphs have exactly the same chords. 
\end{proof}

\section{On the Rado graph} \label{sec rado}

We record the following basic fact about the Rado graph, which may be obvious to some readers:

\begin{observation} \label{Rado}
The Rado graph $R$ is \ILC.
\end{observation}
%%%%%%%%%
\begin{proof}
It is well-known that $R$ is, up to isomorphism, the unique countable graph with the following extension property:
\labtequ{ext}{for every two disjoint finite subsets $U,W$ of $V(R)$, there is a vertex $x\in  V(R) \sm (U\cup W)$ \st\ $N(x) \cap (U\cup W) = U$.}
Pick $v\in V(R)$. Our task is to prove that $R_v$ is isomorphic to $R$, and so it suffices to check that $R_v$ satisfies \eqref{ext} by the aforementioned uniqueness property. So let $U,W$ be disjoint finite subsets of $V(R_v)(=V(R))$. We distinguish two cases: 

If $v\not\in U$, we pick a vertex $w$ satisfying \eqref{ext} in $R$ \wrt\ the sets $U$ and $W':=W \cup \{v\}$. Thus $wv\not\in E(R)$, and therefore $N_R(w)=N_{R_v}(w)$, which means that $w$ satisfies \eqref{ext} in $R_v$ \wrt\ $U$ and $W$.

If $v\in U$, we choose 
\begin{align*}
U'&:=(U \sm N(v)) \cup (W \cap N(v)) \text{ and }\\
W'&:= (U \cap N(v)) \cup (U \sm N(v)),
\end{align*}
and pick a vertex $w$ satisfying \eqref{ext} in $R$ \wrt\ $U'$ and $W'$. Then $wv\in E(R)$ since $v\in U'$, and after the complementation of $N(v)$, we have $N_{R_v}(w) \cap (U \cup W) =U$ as desired. 
\end{proof}

%Recall that a \defi{circle graph} is an intersection graph of a family of chords of $\BS^1$. 

\section{Further problems} \label{sec prob}

It is known that every finite distance-hereditary graph is a circle graph \cite{WikiDHG}. This motivates the following:

\begin{problem} \label{dh}
Is there a strongly universal (countable) distance-hereditary graph? If yes, is there a vertex-transitive one?
\end{problem}

The definition of circle graphs can be generalised  to higher dimensions as follows. A 
%\defi{$\BS^d$-graph} 
\defi{\dsg} is an intersection graph of a subfamily of 
$$\ch_d:= \{ \mathbb{D}^d \cap H \mid H \text{ is a hyperplane of $\R^{d+1}$}\},$$
where the $d$-disc $\mathbb{D}^d$ is the set $\{x\in \R^{d+1} \mid |x,0|\leq 1\}$, and $\BS^d$ is its frontier. Thus $\ch_2$ consists of the flat discs with boundary on $\BS^2$. 

These graph classes are studied in \cite{sphereDim}, where it is proved that \fe\ $d\geq 2$, a countable graph is a $d$-sphere graph \iff\ it is an intersection graph of metric spheres in $\R^d$.

In analogy to \cc, we can define  \defi{the $d$-sphere graph} $\cc^d$ as the intersection graph of the family $\ch$. Notice that $\cc^1=\cc$.

%or \cite{BeSchrLac}; I leave the details to the interested reader.

\medskip

It would be interesting to try to extend the results of this paper to $d>1$:

\begin{problem}
Is there a strongly universal element for the class of countable \dsg s for every $d>1$?
\end{problem}

\begin{problem}
Is $\aut(\BS^d)$ isomorphic to $\aut(\cc^d)$ for every $d>1$?
\end{problem}

%\begin{problem} \label{}
%Is $\cc^d$ \ILC\ for any $d>1$?
%\end{problem}

%The  curve graph $C(S)$ of a closed surface $S$ mentioned in the introduction is defined as follows. Its vertex set comprises the homotopy classes of essential (i.e.\ non homotopically trivial) simple closed curves on $S$. Two such classes are joined with an edge whenever they can be represented by a pair of disjoint curves. This motivates another generalisation of \cc: given $n\in \N$, let $\cc_n$ stand for the intersection graph of the family $\cf$ each element of which is the union of at most $n$ chords of $\BS^1$. Thus 

%\begin{problem} \label{} \end{problem}

\acknowledgement{I thank James Davies for a discussion that triggered this note}. 

\bibliographystyle{plain}
\bibliography{collective}

\end{document}